\documentclass[12pt, twoside]{article}
\usepackage[all, cmtip]{xy}

\usepackage[colorlinks=true]{hyperref}

\usepackage{amsmath,amsthm,amssymb}
\usepackage{times}

\usepackage{times,upref,eucal, mathrsfs, xcolor, dsfont}
\usepackage{amsfonts,amsmath,amstext,amsbsy,amssymb, amsopn,amsthm}

\usepackage{enumerate}

\usepackage{url}

\usepackage{mathtools}
\usepackage{bookmark}

\theoremstyle{plain}
\newtheorem{thm}{Theorem}[section]
\newtheorem{lem}[thm]{Lemma}
\newtheorem{prop}[thm]{Proposition}
\newtheorem{cor}[thm]{Corollary}

\newcommand{\thistheoremname}{}
\newtheorem{genericthm}[thm]{\thistheoremname}

  \newtheorem*{genericthm*}{\thistheoremname}
\newenvironment{namedthm*}[1]
  {\renewcommand{\thistheoremname}{#1}%
   \begin{genericthm*}}
  {\end{genericthm*}}

\theoremstyle{definition}
\newtheorem{defn}{Definition}[section]

\theoremstyle{remark}
\newtheorem{rem}{Remark}[section]

\newcommand{\N}{\mathbb{N}}
\newcommand{\R}{\mathbb{R}}
\newcommand{\C}{\mathbb{C}}
\newcommand{\Z}{\mathbb{Z}}

\newcommand{\T}{\mathbb{T}}

\newcommand{\E}{\mathbb{E}}
\newcommand{\PP}{\mathbb{P}}

\newcommand{\XX}{\mathcal{X}}

\newcommand{\Var}{\mathrm{Var}}

\newcommand{\n}{\|}

\newcommand \spann {{\mathrm{span}}}

\newcommand \Conf {{\mathrm {Conf}}}

\newcommand \an { \text{\,\,and\,\,}}

\newcommand \dist {\mathrm{dist}}
\newcommand \Cov {\mathrm{Cov}}
\newcommand \sine {\mathrm{sine}}

\input xy
\xyoption{all}

\input xy
\xyoption{all}

\def\titlerunning#1{\gdef\titrun{#1}}
\makeatletter
\def\author#1{\gdef\autrun{\def\and{\unskip, }#1}\gdef\@author{#1}}
\def\address#1{{\def\and{\\\hspace*{18pt}}\renewcommand{\thefootnote}{}%
\footnote {#1}}%
\markboth{\autrun}{\titrun}}
\makeatother
\def\email#1{e-mail: #1}
\def\subjclass#1{{\renewcommand{\thefootnote}{}%
\footnote{\emph{Mathematics Subject Classification (2010):} #1}}}
\def\keywords#1{\par\medskip
\noindent\textbf{Keywords.} #1}

\begin{document}

\baselineskip=17pt

\titlerunning{Linear rigidity of stationary stochastic processes}

\title{Linear rigidity of stationary stochastic processes}
\date{}
\author{ Alexander I. Bufetov, Yoann Dabrowski, Yanqi Qiu}
\maketitle

\address{Alexander I. Bufetov: Aix-Marseille Universit{\'e}, Centrale Marseille, CNRS, I2M, UMR7373,  39 Rue F. Joliot Curie 13453, Marseille Cedex 13, France; Steklov Institute of Mathematics, Moscow; Institute for Information Transmission Problems, Moscow; National Research University Higher School of Economics, Moscow; \email{bufetov@mi.ras.ru}}

\address{Yoann Dabrowski: Universit\'e de Lyon, Universit\'e Lyon 1, Institut Camille Jordan, 43 blvd. du 11 novembre 1918,  F-69622 Villeurbanne cedex, France; \email{dabrowski@math.univ-lyon1.fr}}
 
\address{Yanqi Qiu:  Aix-Marseille Universit{\'e}, Centrale Marseille, CNRS, I2M, UMR7373,  39 Rue F. Joliot Curie 13453, Marseille cedex, France; \email{yqi.qiu@gmail.com}}


\begin{abstract}
We consider stationary stochastic processes $\{X_n: n\in \Z\}$ such that   $X_0$ lies in the closed linear span of  $\{X_n: n\neq 0\}$; following Ghosh and Peres, we call such processes  linearly rigid. 
Using a criterion of Kolmogorov, we show that it suffices, for a stationary stochastic process to be linearly rigid, that the spectral density vanish at zero and belong to the Zygmund class $\Lambda_{*}(1)$. We next give sufficient condition for stationary determinantal point processes on $\Z$ and on $\R$ to be linearly rigid. Finally, we show that the determinantal point process on $\R^2$ induced by a tensor square of Dyson sine-kernels is {\it not} linearly rigid. 

\end{abstract}

\section{Introduction}

This paper is devoted to rigidity of stationary determinantal point processes.

Recall that stationary determinantal point processes are strongly chaotic: they have the Kolmogorov property (Lyons \cite{DPP-L}) and the Bernoulli property (Lyons and Steif \cite{Lyons-stationary}); and they satisfy the Central Limit Theorem (Costin and Lebowitz \cite{Costin-Lebowitz}, Soshnikov\cite{DPP-S}). On the other hand, Ghosh \cite{Ghosh-sine} and Ghosh-Peres \cite{Ghosh-rigid} proved, for the determinantal point processes such as Dyson sine process and Ginibre point process, that number of particles in a finite window is measurable with respect to the completion of the sigma-algebra describing the configurations outside that finite window.  Their argument is spectral: they construct, for any small $\varepsilon$, a compactly supported smooth function $\varphi_\varepsilon$, such that  $\varphi_\varepsilon$ equals $1$ in a fixed finite window and the linear statistic corresponding to $\varphi_\varepsilon$ has variance smaller than $\varepsilon$.

In the same spirit, we consider general stationary stochastic processes (in broad sense) $\{X_n: n\in \Z\}$ such that   $X_0$ lies in the closed linear span of  $X_n$, $n\neq 0$; following Ghosh and Peres, we call such processes  linearly rigid. In 1941
Kolmogorov  \cite{kolmogorov-izv}, \cite{kolmogorov-vestnik} gave a  sufficient condition for linear rigidity: namely,  that the spectral density of our process vanish at zero and the integral of the  inverse of the spectral 
density diverge. Such a condition is easy to verify for  example for the sine-process, since the spectral density  $\omega$ in the neighbourhood of zero has the form $\omega(\theta)=|\theta|$. 
More generally,  in order that  a stationary stochastic process  be rigid, we check that it suffices  that the spectral density vanish at zero and belong to the Zygmund class $\Lambda_{*}(1)$. 
We next give sufficient condition for stationary determinantal point processes on $\Z$ and on $\R$ to be rigid. Finally, we show that the determinantal point process on $\R^2$ induced by a tensor square of Dyson sine-kernel is {\it not} linearly rigid.

We now turn to more precise statements.  Let $X = \{X_n : n \in \Z^d\}$ be a multi-dimensional time stationary  stochastic process of real-valued random variables defined on a probability space $(\Omega, \PP)$. Let $H(X) \subset L^2(\Omega, \PP)$ denote the closed subspace linearly spanned by $\{X_n : n \in \Z^d \}$ and let $\check{H}_0(X)$ denote the one linearly spanned by $\{X_n : n \in \Z^d \setminus \{ 0\}\}$. 
 
 \begin{defn}
  The stochastic process $X$ is said to be linearly rigid if 
 \begin{align}
 X_0 \in \check{H}_0(X).
 \end{align}
 \end{defn}

 Let $\Conf(\R^d)$ be the set of locally finite configurations on $\R^d$. 
For a bounded Borel subset $B \subset \R^d$, we denote $N_B: \Conf(\R^d) \rightarrow \N \cup \{0\}$ the function defined by 
 $$
 N_B(\XX) : = \text{\, the cardinality of $B \cap \XX$ }.
 $$ 
 The space $\Conf(\R^d)$ is equipped with the Borel $\sigma$-algebra which is the smallest $\sigma$-algebra making all $N_B$'s measurable. Recall that a point process with phase space $\R^d$ is, by definition, a Borel probability measure on the space $\Conf(\R^d)$. For the background on point process, the reader is referred to Daley and Vere-Jones' book \cite{Daley}.

 Given a stationary point process on $\R^d$ and $\lambda>0$, we introduce the  stationary stochastic process $N^{(\lambda)} = (N_n^{(\lambda)})_{n\in \Z^d}$  by  the formula
\begin{align}\label{def-N}
N_{n}^{(\lambda)}(\XX): =\text{\, the cardinality of $\mathcal{X} \cap\left( n\lambda +  [- \lambda/2, \lambda/2  )^d\right)$.}
\end{align}

\begin{defn}\label{window}
A stationary point process $\PP$ on $\R^d$ is called {\bf linearly rigid}, if for any $\lambda>0$, the stationary stochastic process $N^{(\lambda)} = (N_n^{(\lambda)})_{n\in \Z^d}$ is linearly rigid, i.e., 
$$
N_0^{(\lambda)} \in \check{H}_0(N^{(\lambda)}).
$$
\end{defn}

\medskip

The above  definition is motivated by the definition due to Ghosh and Peres of rigidity of point processes on $\R^d$, see \cite{Ghosh-sine} and \cite{Ghosh-rigid}.  Given a Borel subset $C \subset \R^d$, we will denote 
 $$
 \mathcal{F}_C  = \sigma(\{N_B:  B \subset C, B \text{ bounded Borel}\})
 $$
 the $\sigma$-algebra generated by all  random variables of the form $N_B$ where $B \subset C$ ranges over all bounded Borel subsets of $C$. Let $\PP$ be a point process on $\R$, i.e., $\PP$ is  a Borel probability on $\Conf(\R^d)$,  and denote $\mathcal{F}_C^{\PP}$ for the $\PP$-completion of $\mathcal{F}_C$.

\begin{defn}[Ghosh \cite{Ghosh-sine}, Ghosh-Peres \cite{Ghosh-rigid}]\label{rigid-defn}
A point process $\PP$ on $\R^d$  is called {\bf number rigid}, if for any bounded Borel set $B \subset \R^d$ with Lebesgue-negligible boundary $\partial B$, the random variable $N_B $
is $\mathcal{F}_{\R^d \setminus B}^\PP$-measurable. 
\end{defn}

\begin{rem}
Of course, in the above definition, it suffices to take Borel sets $B$  of the form $[-\gamma, \gamma)^d$ for $\gamma >Â 0$, cf. \cite{Ghosh-rigid}. 
\end{rem}

A linear rigid stationary point process on $\R^d$ is of course rigid in the sense of Ghosh and Peres. Observe that proofs for rigidity in \cite{Ghosh-sine}, \cite{Ghosh-rigid} and \cite{Buf-rigid} in fact establish linear rigidity.  We would like also to mention a notion of insertion-deletion tolerance studied by  Holroyd and Soo in \cite{Soo}, which is in contrast to the notion of rigidity property.

\section{The Kolmogorov criterion for linear rigidity} 
 In this note, the Fourier transform of a function $f: \R^d \rightarrow \C$ is defined as 
 $$
 \widehat{f}(\xi)= \int_{\R^d} f(x) e^{- i 2 \pi x \cdot \xi} dx.
 $$
 Denote by $\T^d = \R^d/\Z^d$ the $d$-dimensional torus. In what follows, we identify $\T^d$ with $[-1/2, 1/2)^d$. The Fourier coefficients of a measure $\mu$ on $\T^d$ are given, for any $k\in \Z^d$, by the formula 
 $$
 \hat{\mu}(k) = \int_{\T^d} e^{- i 2 \pi k \cdot \theta}d\mu_X( \theta), \text{ where } k\cdot \theta: = k_1 \theta_1 + \cdots + k_d \theta_d.
 $$

 Denote by $\mu_X$ the spectral measure of $X$, i.e.,
\begin{align}\label{sp-m}
\forall k \in \Z^d, \quad \E(X_0X_{k}) = \E(X_nX_{n+k})=  \int_{\T^d} e^{- i 2 \pi k \cdot \theta}d\mu_X( \theta) = \hat{\mu}_X(k).
\end{align}

Recall that we have the following natural isometric isomorphism
\begin{align}\label{iso}
H(X) \simeq  L^2 (\T^d, \mu_X),
\end{align}
by assigning to $X_n\in H(X)$ the function $\theta \mapsto e^{i 2 \pi n \cdot \theta} \in L^2(\T^d, \mu_X)$. 

Let $\mu_X = \mu_a + \mu_s$ be the Lebesgue decomposition of $\mu_X$ with respect to the normalized Lebesgue measure $m(d\theta) = d\theta_1 \cdots d\theta_d$ on $\T^d$, i.e., $\mu_a$ is absolutely continuous with respect to $m$ and $\mu_s$ is singular to $m$.
Set 
$$
\omega_X (\theta)  : = \frac{d\mu_a}{dm} (\theta).
$$

\begin{lem}[The Kolmogorov Criterion ]\label{K}
We have 
$$
\dist (X_0,  \check{H}_0(X)) =   \left(\int_{\T^d} \omega_X^{-1} dm\right)^{-1/2},
$$
where by $\dist (X_0,  \check{H}_0(X))$ we mean the least $L^2$-distance between the random variable $X_0$ and the linear space $\check{H}_0(X)$ and the right side is to be interpreted as zero if 
$
\int_{\T^d}  \omega_X^{-1} dm  = \infty.
$
\end{lem}

\begin{cor}\label{prop-l-rig}
The stationary stochastic process $X= (X_n)_{n\in \Z^d}$ is linearly rigid if and only if 
$$
\int_{\T^d}  \omega_X^{-1} dm  = \infty.
$$
\end{cor}

Lemma \ref{K} is due to Kolmogorov \cite{kolmogorov-izv}, \cite{kolmogorov-vestnik}. 
 For the reader's convenience, we include its proof. 

\begin{proof}[Proof of Lemma \ref{K}] We follow the argument of Lyons-Steif \cite{Lyons-stationary}.
By the Lebesgue decomposition of $\mu$, we may take a subset $A \subset \T^d$ of full Lebesgue measure $m(A) =1$, such that $\mu_a(A) =1$ and $\mu_s(A) = 0$.

Denote $$L_0  = \overline{\spann}^{L^2(\T^d, \mu_X)} [e^{i 2 \pi n \cdot \theta}: n \ne 0].$$
By  the isometric isomorphism \eqref{iso}, it suffices to show that 
\begin{align}\label{distance}
\dist (1, L_0) =   \left(\int_{\T^d} \omega_X^{-1} dm\right)^{-1/2},
\end{align}
where $1$ is the constant function taking value $1$. 
Write
$$
1  =  p + h, \text{\, such that \,} p \perp L_0,  h \in L_0.
$$
Modifying, if necessary,  the values of $p$ and $h$ on a $\mu$-negligible subset, we may assume that 
$$
1 = p( \theta) + h(\theta) \text{\, for all $\theta\in \T^d$.}
$$
Since $p\perp L_0$,  we have 
\begin{align}\label{fourier}
0 = \langle p, e^{i 2 \pi n \cdot \theta} \rangle_{L^2(d\mu)} = \int_{\T^d} p(\theta) e^{- i 2 \pi n \cdot \theta} d\mu(\theta), \text{\, for any $n \in \Z^d \setminus \{0\}$.}
\end{align}
Let $\xi\in \C$ denote 
$$
\xi =  \int_{\T^d} p(\theta) d\mu(\theta). 
$$
Then by \eqref{fourier}, all the Fourier coefficients of the complex measure $ p\cdot d\mu$ coincide with the corresponding Fourier coefficients of $\xi dm$ (the multiple of Lebesgue measure $dm$ by $\xi$), consequently, we have
$$
p \cdot d \mu  =  \xi dm.
$$
It follows that $p$ must vanish almost everywhere with respect to the singular component $\mu_s$ of $\mu$, and $p(\theta) \omega_X (\theta) =\xi$  for $m$-almost every $\theta\in \T^d$.  Thus we have 
\begin{align}\label{p-l2-norm}
\n p \n_{L^2(d\mu)} = \n p \n_{L^2(d\mu_a)},
\end{align}
and
\begin{align}\label{h-value}
h(\theta) = 1 - \xi  \omega_X(\theta)^{-1} \text{\, for $m$-almost every $\theta\in \T^d$.}
\end{align}

\medskip

{\bf Case 1:} 
$\int_{\T^d}  \omega_X^{-1} dm < \infty.$ 

 Define a function $f: \T^d \rightarrow \C$ by $f = \omega_X^{-1} \chi_A$.  Then 
$
f \in L^2(d\mu)\ominus L_0.
$ 
Indeed, 
$$
\n f \n_{L^2(d\mu)}^2 = \int_{\T^d}  \omega_X^{-2} \chi_{A} d\mu = \int_{\T^d}  \omega_X^{-2} d\mu_a = \int_{\T^d}  \omega_X^{-1} dm < \infty.
$$
And, for all $n \in \Z^d\setminus 0$,  
$$
\langle  f, e^{i 2 \pi n \cdot \theta} \rangle_{L^2(d\mu)} = \int_{\T^d}  \omega_X(\theta)^{-1} \chi_{A}(\theta) e^{-i 2 \pi n \cdot \theta}  d\mu(\theta)= \int_{\T^d}  e^{-i 2 \pi n \cdot \theta}  dm(\theta) = 0.
$$
It follows that $f \perp h$, i.e.,
$$
0 = \langle h, f \rangle_{L^2(d\mu)} = \int_{\T^d}  h   \omega_X^{-1} \chi_A d\mu  = \int_{\T^d}  h dm.
$$
By \eqref{h-value}, we get 
$$
\int_{\T^d}  (1 - \xi \omega_X^{-1} )  dm=0,
$$
and hence
$$
\xi =  \Big(\int_{\T^d} \omega_X^{-1} dm\Big)^{-1}.
$$
It follows that 
\begin{align*}
\dist (1, L_0)^2 = \n p \n_{L^2(d\mu)}^2 = \n p\n_{L^2(d\mu_a)}^2 = \xi^2 \int_{\T^d}  \omega_X^{-2} \omega_X dm = \xi.
\end{align*}
This shows the desired equality \eqref{distance}.

\medskip

{\bf Case 2:}  $\int_{\T^d}  \omega_X^{-1} dm = \infty.$ 

We claim that $\xi=0$. If the claim were verified, then we would get the desired identity in this case
$$
\dist (1, L_0) = 0.
$$
So let us turn to the proof of the claim. We argue by contradiction. If $\xi \ne 0$, then $p \ne 0$ and 
$$
\n p \n_{L^2(d\mu)}^2  = \n p\n_{L^2(d\mu_a)}^2 = \xi^2 \n \omega_X^{-1} \n_{L^2(d\mu_a)}^2 = \xi^2  \int_{\T^d}  \omega_X^{-1} dm  = \infty.
$$ 
This contradicts the fact that $p \in L^2(d\mu)$.
\end{proof}

 \begin{rem}
 If the spectral measure $\mu_X$ is absolutely continuous and given by $\mu_X(dz) = \omega (z) dm(z)$, then for any $n \in \N$, the following are equivalent: 
 \begin{enumerate}
 \item[$(i)$] $\sum_{l = -n}^n X_l \in \overline{\spann}^{H(X)} \left\{ X_k : k \in \Z, | k | \ge n +1\right\}  $.
 \item[$(i)'$] $\sum_{ l  = - n }^n z^l \in \overline{\spann}^{L_{\omega}^2}\{z^l: k \in \Z, | k | \ge n + 1 \}$.
 \item[$(ii)$] For any $w_1, w_2, \dots, w_n \in \C\setminus \{1\}$, $$\int_{\T} \frac{\prod_{l = 1}^n  | (z - w_l) (z - \bar{w}_l)|^2}{\omega(z)} d m(z) = \infty. $$
 \item[$(ii)'$] For any $w_1, w_2, \dots, w_n\in \T \setminus\{1\},$  $$\int_{\T} \frac{\prod_{l = 1}^n  | (z - w_l) (z - \bar{w}_l)|^2}{\omega(z)} d m(z) = \infty. $$
 \end{enumerate}

Indeed, $(i)$ and $(i)'$ are equivalent.  Assume $(i)'$ is satisfied, let us show $(ii)$. If $(ii)$ is violated, then there exist $w_1, w_2,\dots, w_n\in \C\setminus\{1\}$, such that $$\int_{\T} \frac{\prod_{l = 1}^n  | (z - w_l) (z - \bar{w}_l)|^2}{\omega_X(z)} d m(z) < \infty. $$ Define $$h(z) : =  \frac{ \prod_{ l = 1}^n( z-w_l)(z - \bar{w}_l)}{z^n\omega_X(z)} = \frac{\sum_{l = -n}^n a_l z^l}{\omega_X(z)}.$$ 
Then $h \in L_\omega^2(\T) \ominus \overline{\spann}^{L^2_\omega}\{z^k: k \in \Z, | k | \ge n +1\}$. We have \begin{align*}\Big(\sum_{l = -n}^n z^l, h(z) \Big)_{L_\omega^2} & = \sum_{l = -n}^n a_l  = \prod_{l = 1}^n | 1 - w_l|^2 \ne 0. \end{align*}  This contradicts $(i)'$, hence $(i)'$ implies $(ii)$. 
 
 Conversely, let us assume $(ii)$ and show $(i)'$. If $(i)'$ is not satisfied, then there exists a function $g \in L_\omega^2 \ominus \overline{\spann}^{L_\omega^2}\{ z^l: k \in \Z, | k | \ge n + 1 \}$, such that $g \ne 0$ and the scalar product $ (\sum_{l = -n }^n z^l, g)_{L_\omega^2} \ne 0.$ We have  $$ 0 = \int_{\T} g(z) z^k \omega(z) dm(z), \text{ for any $k \in \Z$, $| k | \ge n + 1$. }$$ This implies that there exists $ (c_{-n}, \dots, c_n)$ such that $$g(z) \omega(z) = \sum_{-n}^n c_l z^l.$$ Hence $g(z)   = \frac{ \sum_{ -n}^n c_l z^l}{\omega(z)}$ and $$\sum_{-n}^n c_l = \Big(g, \sum_{-n}^n z^l\Big)_{L_\omega^2} \ne 0 .$$ Since $\omega(z) = \omega(z^{-1})$, if we denote $\check{g}(z) : = g(z^{-1})$, then $\check{g} \in L_\omega^2(\T)$. Thus $\Re (g + \check{g})$ and $\Im (g  + \check{g})$ are functions in $L_\omega^2(\T).$ We have $$\Re (g + \check{g})(z) = \frac{ \sum_{-n}^n  \Re(c_l) (z^l + z^{-l})}{\omega(z)} \text{ \, and  \, } \Im(g + \check{g})(z) = \frac{\sum_{-n}^n \Im(c_l) ( z^l+ z^{-l})}{\omega(z)}.$$ Since $\sum_{-n}^n c_l \ne 0$, we may assume without loss of generality that $\sum_{-n}^n \Re (c_l) \ne 0$. Define $P(z)$ the polynomial given by $P(z) = z^n \sum_{-n}^n \Re(c_l) (z^l + z^{-l})$ and let $m = \deg P \le m $ then there exist $w_1, \dots, w_m$ such that  $$P(z) = \Re(c_m) \prod_{l = 1}^m (z - w_l) (z - \bar{w}_l).$$ Since $P(1) = \sum_{-n}^n \Re(c_l) \ne 0$, we know that $w_1, \dots, w_m$ are all different from 1. Now using the fact $\Re( g + \check{g}) \in L_\omega^2$, we deduce that $$ \int_\T \frac{\prod_{l  = 1}^m | (z-w_l)(z- \bar{w}_l)|^2}{\omega(z)} <Â \infty,$$ which of course violates $(ii)$. This contradiction shows that $(ii)$ implies $(i)'$. 

The equivalence between $(ii)$ and $(ii)'$ is obvious. 
 
\end{rem}

Denote by $\Cov (U, V)$  the covariance between two random variables $U$ and $V$: 
$
\Cov (U, V)   = \E(UV) -  \E(U) \E(V).
$

If $X  = (X_n)_{n\in \Z^d}$ is a stochastic process such that 
\begin{align}\label{cov-sum}
\sum_{n\in \Z^d} | \Cov(X_0, X_n)| < \infty,
\end{align}
then we may define a continuous function on $\T^d$ by the formula
\begin{align}\label{density}
 \omega_X(\theta): =    \sum_{n\in \Z^d} \Cov(X_0, X_n) e^{i 2\pi n \cdot \theta}.
\end{align}

\begin{lem}\label{lem-density}
Let  $X  = (X_n)_{n\in \Z^d}$ be a stationary stochastic process satisfying condition \eqref{cov-sum}.
Then we have the following explicit Lebesgue decomposition of $\mu_X$: 
\begin{align}\label{measure-dec}
\mu_X = (\E X_0)^2 \cdot \delta_0 + \omega_X \cdot m, 
\end{align}
where $\delta_0$ is the Dirac measure on the point $0 \in \T^d$ and $\omega_X$ is the function on $\T^d$ defined by \eqref{density}.
\end{lem}

\begin{proof}
Note that, under the assumption \eqref{cov-sum}, the function $\omega_X(\theta)$ is well-defined and  continuous on $\T^d$. For proving the decomposition \eqref{measure-dec}, it suffices to show that the Fourier coefficients of $\mu_X$ coincide with those of $\nu_X: = (\E X_0)^2 \cdot \delta_0 + \omega_X \cdot m$. But if $n \in \Z^d$, then 
\begin{align*}
\hat{\nu}_X(n) = (\E X_0)^2 + \Cov(X_0, X_n)  = \E(X_0X_n)  = \hat{\mu}_X(n). 
\end{align*}
The lemma is completely proved.
\end{proof}

\section{A sufficient condition for linear rigidity}

\begin{thm}\label{l-rigid}
Let  $X  = (X_n)_{n\in \Z}$ be a stationary stochastic process. If
\begin{align}\label{summable}
\sup_{N \ge 1 }  \left ( N \sum_{|n| \ge N}  | \Cov(X_0, X_n)| \right) < \infty,
\end{align}
and 
\begin{align}\label{cov-cond}
\sum_{n \in \Z}  \Cov(X_0, X_n) = 0.
\end{align}
Then  $X$ is linearly rigid.
\end{thm}

\begin{rem}
The condition \eqref{summable} is a sufficient condition such that the spectral density $\omega_X$ is a function in the Zygmund class $\Lambda_{*}(1)$, see below for definition.  The condition \eqref{cov-cond} implies in particular that  $\omega_X$ vanishes at the point $0 \in \T$.
\end{rem}

We shall apply a result of F. M\'oricz \cite[Thm. 3]{Moricz} on absolutely convergent Fourier series and Zygmund class functions. Recall that a continuous $1$-periodic function $\varphi$ defined on $\R$ is said to be in the {\it Zygmund class} $\Lambda_{*}(1)$, if there exists a constant $C$ such that 
\begin{align}\label{A}
| \varphi (x+h) - 2 \varphi(x)  + \varphi(x-h) | \le C h
\end{align}
for all $x\in \R$ and for all $h>0$.

\begin{thm}[M\'oricz, \cite{Moricz}]\label{MM}
If $\{c_n\}_{n\in \Z} \in \C $ is such that
\begin{align}\label{big-O}
\sup_{N \ge 1} \left( N \sum_{|n| \ge N}  | c_n|\right) < \infty, 
\end{align}
then the function $\varphi (\theta)= \sum_{n\in \Z} c_n e^{i 2 \pi n \theta}$ is in the Zygmund class $\Lambda_{*}(1)$.
\end{thm}

\begin{proof}[Proof of Theorem \ref{l-rigid}]
First, in view of \eqref{density}, our assumption \eqref{cov-cond} implies 
\begin{align*}
 \omega_X(0)  =0.
\end{align*}

Next, by Theorem \ref{MM}, under the assumption \eqref{summable}, we have 
\begin{align*}
\omega_X  \in \Lambda_{*}(1).
\end{align*}
Since all Fourier coefficients of $\omega_X$ are real, we have 
$$
\omega_X(\theta) = \omega_X(-\theta).
$$ 
Consequently, there exists $C>0$, such that 
\begin{align*}
 \omega_X(\theta)  = \frac{\omega_X(\theta) + \omega_X(-\theta)}{2} 
 =\frac{\omega_X( \theta) + \omega_X(-  \theta )  - 2 \omega_X(0)}{2}\le C |  \theta|,
\end{align*}
whence
$$
\int_\T \omega_X^{-1} dm = \infty,
$$
and the stochastic process $X= (X_n)_{n\in \Z}$ is linearly rigid by the Kolmogorov criterion. 
\end{proof}

\section{Applications to stationary determinantal point processes}

In this section, we first give a sufficient condition for linear rigidity of stationary determinantal point processes on $\R$ and then give an example of a very simple stationary, but not linearly rigid, determinantal point process on $\R^2$. 
We briefly recall the main definitions.
 Let $B\subset \R^d$ be a bounded Borel subset. Let $K_B: L^2(\R^d) \rightarrow L^2(\R^d)$ be the  operator of convolution with the  Fourier transform $\widehat{\chi_B}$ of the indicator function $\chi_B$. In other words, the kernel of $K_B$ is 
\begin{align}\label{kb-kernel}
K_{B}(x,y) = \widehat{\chi_B}(x-y).
\end{align}
In particular, if $d = 1$ and $B  = (-1/2, 1/2)$, then we find the well-known Dyson sine kernel 
$$
K_{\sine}(x,y) = \frac{\sin( \pi (x-y))}{\pi (x-y)}.
$$ 
Note that we always have $K_B(x,x)= K_B(0,0).$
 
Denote by $\PP_{K_B}$ the determinantal point process induced by $K_B$. For the background on  the determinantal point processes, the reader is referred to  \cite{DPP-HKPV}, \cite{DPP-L}, \cite{DPP-M}, \cite{DPP-S}.

\begin{prop}\label{prop-rep}
Let $\PP_{K_B}$ be the stationary determinantal point process on $\R^d$ induced by the kernel $K_B$ in \eqref{kb-kernel}. For any $\lambda>0$, denote by $N^{(\lambda)} = (N_n^{(\lambda)})_{n\in \Z^d}$  the stationary stochastic process associated to $\PP_{K_B}$ as in \eqref{def-N}. Then 
\begin{align}\label{ab-cv}
\sum_{n\in \Z^d} |\Cov(N_0^{(\lambda)}, N_n^{(\lambda)})| < \infty
\end{align}
and
\begin{align}\label{only-root}
\sum_{n\in \Z^d} \Cov(N_0^{(\lambda)}, N_n^{(\lambda)})  =0.
\end{align}
\end{prop}

\begin{proof}
Fix a number $\lambda >0$, for simplifying the notation, let us denote $N_n^{(\lambda)}$ by $N_n$. Denote for any $n\in \Z^d$, 
$$
Q_n = n \lambda + [-\lambda/2, \lambda/2)^d.
$$
By definition of a determinantal point process, we have 
$$
\E(N_n)= \E(N_0) = \int_{Q_0} K_B(x,x) dx  = \lambda^d K_B(0,0).
$$
If $n \ne 0$, we have 
\begin{align*}
\E (N_0 N_n)
& 
= \iint   \chi_{Q_0}(x) \chi_{Q_n} (y) \left|  \begin{array}{cc}  K_B(x,x) & K_B(x,y) \\ K_B(y,x) & K_B(y,y)\end{array}  \right| dxdy 
\\
& = \lambda^{2d} K_B(0,0)^2 - \iint_{Q_0\times Q_n} | K_B(x,y)|^2 dxdy,
\end{align*}
whence
\begin{align}\label{cor-form}
\Cov(N_0, N_n) = -\iint_{Q_0\times Q_n} | K_B(x,y)|^2 dxdy.
\end{align}
We also have
\begin{align*}
&\E(N_0^2) 
=  
\E\left[  \sum_{x, y \in \mathcal{X}} \chi_{Q_0} (x) \chi_{Q_0} (y)  \right] 
\\
& = \E\left[ \sum_{x \in \mathcal{X}} \chi_{Q_0} (x) \right] +  \E \left[\sum_{x, y \in \mathcal{X}, x \ne y} \chi_{Q_0} (x) \chi_{Q_0} (y)   \right]
\\ 
& = \int_{Q_0} K_B(x,x)dx  +
  \iint   \chi_{Q_0}(x) \chi_{Q_0} (y) \left|  \begin{array}{cc}  K_B(x,x) & K_B(x,y) \\ K_B(y,x) & K_B(y,y)\end{array}  \right| dxdy 
  \\ 
  & = \lambda^d K_B(0,0) + \lambda^{2d} K_B(0,0)^2  - \iint_{Q_0\times Q_0} | K_B(x,y)|^2 dxdy,
\end{align*}
whence 
\begin{align}\label{cor-form2}
\Cov(N_0, N_0) = \Var(N_0) = \lambda^d K_B(0,0) - \iint_{Q_0 \times Q_0} | K_B(x,y)|^2 dxdy.
\end{align}
Now recall that  $K_B$ is an orthogonal projection. Thus we have 
\begin{align}\label{rep-kb}
K_B(0,0) = K_B(x,x)  = \int |K_B(x,y)|^2 dy = \sum_{n \in \Z^d} \int_{Q_n} | K_B(x,y)|^2dy.
\end{align}
The identities \eqref{cor-form}, \eqref{cor-form2} and \eqref{rep-kb} imply that 
\begin{align*}
\sum_{n\in \Z^d} \Cov(N_0, N_n) & = \lambda^d K_B(0,0) - \int_{Q_0} dx \sum_{n\in \Z^d} \int_{Q_n} | K_B(x,y)|^2dy 
\\
& =  \lambda^d K_B(0,0) -  \lambda^d K_B(0,0) =0.
\end{align*}
Moreover, the above series converge absolutely. Proposition  \ref{prop-rep} is completely proved.
\end{proof}

\begin{cor}\label{cor-cont-0}
The spectral density $\omega_{N^{(\lambda)}}$ of the stochastic process $N^{(\lambda)} = (N_n^{(\lambda)})_{n\in\Z^d}$ is a continuous non-negative function on $\T^d = [-\frac{1}{2}, \frac{1}{2}]^d$ and  vanishes only at $(0, \cdots, 0)$. 
\end{cor}

\begin{proof}
By Lemma \ref{lem-density}, the spectral density $\omega_{N^{(\lambda)}}$ of the stochastic process $N^{(\lambda)}$  is given by 
\begin{align}\label{s-N-d}
\omega_{N^{(\lambda)}}(\theta_1, \cdots, \theta_d)  = \sum_{n\in \Z^d} \Cov(N_0^{(\lambda)}, N_n^{(\lambda)}) e^{i 2\pi (n_1\theta_1 + \cdots+ n_d\theta_d)}. 
\end{align}
By \eqref{ab-cv}, the series in  \eqref{s-N-d} converges uniformly and absolutely on $\T^d$. It follows that $\omega_{N^{(\lambda)}}$  is a continuous function on $\T^d$.

Now the equality \eqref{only-root} implies that $\omega_{N^{(\lambda)}}(0, \cdots, 0)=0$.  Moreover,   for any $\theta = (\theta_1, \cdots, \theta_d) \in \T^d \setminus\{(0, \cdots, 0)\}$, we have 
\begin{align*}
|\sum_{n\in \Z^d\setminus\{0\}} \Cov(N_0^{(\lambda)}, N_n^{(\lambda)}) e^{i 2\pi (n_1\theta_1 + \cdots+ n_d\theta_d)}| < \sum_{n\in \Z^d\setminus\{0\}} |\Cov(N_0^{(\lambda)}, N_n^{(\lambda)})|.
\end{align*}
By \eqref{cor-form}, we have 
$$
|\Cov(N_0^{(\lambda)}, N_n^{(\lambda)})| = - \Cov(N_0^{(\lambda)}, N_n^{(\lambda)}) \text{for any $n\in\Z^d\setminus\{0\}$.}
$$ 
Note that if $\theta = (\theta_1, \cdots, \theta_d) \ne (0, \cdots, 0)$, then 
\begin{align*}
\omega_{N^{(\lambda)}}(\theta_1, \cdots, \theta_d)&\ge   \Cov(N_0^{(\lambda)}, N_0^{(\lambda)})  - |  \sum_{n\in \Z^d\setminus\{0\}} \Cov(N_0^{(\lambda)}, N_n^{(\lambda)}) e^{i 2\pi (n_1\theta_1 + \cdots+ n_d\theta_d)}|
\\
&>\Cov(N_0^{(\lambda)}, N_0^{(\lambda)})  -   \sum_{n\in \Z^d\setminus\{0\}} |\Cov(N_0^{(\lambda)}, N_n^{(\lambda)})|
\\
& = \sum_{n\in \Z^d} \Cov(N_0^{(\lambda)}, N_n^{(\lambda)}) =0.
\end{align*}
This shows that $\omega_{N^{(\lambda)}}$ vanishes only at $(0, \cdots, 0)$.
\end{proof}

\subsection{Stationary determinantal point processes on $\R$}

\begin{thm}\label{DPP-rigid-bis}
Assume that $B\subset \R$ satisfies
\begin{align}\label{fourier-bdd}
\sup_{R >0}\left( R \int_{|\xi| \ge R } |  \widehat{\chi_B}(\xi)|^2 d\xi\right)<\infty.
\end{align}
Then the stationary determinantal point process $\PP_{K_B}$ is linearly rigid. 
\end{thm}

\begin{proof}
By definition of linear rigidity, we need to show that  for any $\lambda>0$, the stochastic process $N^{(\lambda)} = (N_n^{(\lambda)})_{n\in \Z}$ is linearly rigid. As in the proof of Proposition \ref{prop-rep}, we denote $N_n^{(\lambda)}$ by $N_n$.  By  Theorem \ref{l-rigid},  it suffices to show that
\begin{align}\label{v1}
 \sup_{N \ge 1} \left( N \sum_{| n | \ge N}|\Cov(N_0, N_n)|  \right) <Â \infty,
\end{align}
and 
\begin{align}\label{v2}
\sum_{n\in \Z} \Cov(N_0, N_n)  =0.
\end{align}
By Proposition \ref{prop-rep}, the identity \eqref{v2} holds in the general case. It remains to prove \eqref{v1}. 
By \eqref{cor-form}, we have 
\begin{align*}
& \sup_{N\ge 1 } \left( N \sum_{|n| \ge N } | \Cov(N_0, N_n) |\right)  = \sup_{N\ge 1 } N \int_{x\in Q_0}\int_{y\in \bigcup\limits_{|n| \ge N} Q_n} | \widehat{\chi}_B(x-y)|^2dxdy
\\
 =&  \sup_{N\ge 1 } N \int_{-\lambda/2}^{\lambda/2} \int_{|y| \ge (N-1/2)\lambda} | \widehat{\chi}_B(x-y)|^2dxdy 
 \\
 \le & \sup_{N\ge 1 } N \int_{-\lambda/2}^{\lambda/2} \int_{|\xi | \ge (N-1)\lambda} | \widehat{\chi}_B(\xi)|^2dxdy 
=  \sup_{N\ge 1 } \lambda N  \int_{|\xi|\ge (N-1)\lambda } | \widehat{\chi}_B(\xi)|^2d\xi< \infty,
 \end{align*}
where in the last inequality, we used our assumption  \eqref{fourier-bdd}. Theorem \ref{DPP-rigid-bis} is proved completely.
 \end{proof}

\begin{rem} 
When $B$ is a finite union of finite intervals on the real line, the rigidity of the stationary determinantal point 
process $\PP_{K_B}$ is due to Ghosh \cite{Ghosh-sine}. 
\end{rem}

\subsection{Tensor product of sine kernels}
In higher dimension, the situation becomes quite different. Let 
$$
S= I \times I =  (-1/2, 1/2)\times (-1/2, 1/2) \subset \R^2.
$$ 
Then the associate kernel $K_S$ has a tensor form: $
K_S = K_\sine \otimes K_\sine,
$
that is, for $x= (x_1, x_2)$ and $y= (y_1, y_2)$ in $\R^2$, we have
$$
K_S(x,y) = K_{\sine}(x_1, y_1) K_\sine(x_2,y_2) = \frac{ \sin(\pi(x_1-y_1))}{\pi(x_1-y_1)}  \frac{ \sin(\pi(x_2-y_2))}{\pi(x_2-y_2)}.
$$

\begin{prop}\label{counter-ex}
The determinantal point process $\PP_{K_S}$ is not linearly rigid. More precisely,  let  $N^{(1)} = (N_n^{(1)})_{n\in \Z^2}$ be the stationary stochastic process given as in Definition \ref{window}, then 
$$
N_0^{(1)} \notin \check{H}_0(N^{(1)}).
$$
\end{prop}

\bigskip

To prove the above result, we need to introduce some extra notation. First, we define the multiple Zygmund class $\Lambda_{*}$ as follows. A continuous function $\varphi(x,y)$ periodic in each variable with period $1$ is said to be in the multiple Zygmund class $\Lambda_{*}(1,1)$ if for the double difference difference operator $\Delta_{2,2}$ of second order in each variable, applied to $\varphi$, there exists a constant $C>0$, such that for all $x= (x_1,x_2) \in (-1/2,1/2)\times (-1/2, 1/2)$ and $h_1, h_2 >0$, we have
\begin{align}\label{double-d}
| \Delta_{2,2} \varphi(x_1,x_2; h_1, h_2)|\le C h_1 h_2,
\end{align}
where 
\begin{align*}
& \Delta_{2,2} \varphi(x_1, x_2; h_1, h_2)
: =  \varphi (x_1 + h_1, x_2 + h_2) + \varphi (x_1 - h_1, x_2 + h_2) 
\\
& \quad + \varphi (x_1 + h_1, x_2 - h_2) + \varphi (x_1 - h_1, x_2 - h_2) - 2 \varphi (x_1 + h_1, x_2) 
\\
 & - 2 \varphi (x_1-h_1, x_2 ) - 2 \varphi (x_1, x_2 + h_2) - 2 \varphi (x_1, x_2 - h_2) + 4 \varphi (x_1, x_2 ).
\end{align*}

The following result is due to F\"ul\"op and M\'oricz \cite[Thm 2.1 and Rem. 2.3]{Fulop-Moricz}
\begin{thm}[F\"ul\"op-M\'oricz]\label{FM}
If $\{c_{jk}\}_{j, k \in \Z} \in \C $ is such that
\begin{align}\label{big-O}
\sup_{N\ge1, M \ge 1} \left( M N \sum_{|j| \ge N, |k| \ge M}  | c_{jk}| \right) < \infty, 
\end{align}
then the function 
$$
\varphi (\theta_1, \theta_2)= \sum_{j, k \in \Z} c_{jk} e^{i 2 \pi (j \theta_1 + k\theta_2)}
$$ is in the Zygmund class $\Lambda_{*}(1,1)$.
\end{thm}

\bigskip

Let us turn to the study of the density function $\omega_{N^{(1)}}$.

\begin{lem}\label{lem-lower}
There exists $c>0$, such that for any $\theta_1, \theta_2 \in [-1/2, 1/2]$, we have 
\begin{align}\label{low-bdd}
\omega_{N^{(1)}}(\theta_1, \theta_2) \ge c(|\theta_1| + | \theta_2|).
\end{align}
\end{lem}
\begin{proof}
To make notation lighter, in this proof we simply  write  $\omega$ for $\omega_{N^{(1)}}$.  

For any $n   = (n_1, n_2) \in\Z^2$,  let us denote $S_n=S \times (n+S)$ where  
$$
 n+S: = (-1/2+n_1, 1/2+n_1)\times (-1/2+n_2, 1/2+n_2).
 $$ 
 By
the same argument as in the proof of Proposition  \ref{prop-rep}, we obtain that for any $n = (n_1, n_2) \in \Z^2 \setminus \{0\}$, 
\begin{align}\label{neg-non-ori}
\widehat{\omega}(n)= -\int_{S_n} | K_S(x,y)|^2 dxdy,
\end{align}
and 
\begin{align*}
\widehat{\omega}(0) = K_S(0,0) - \int_{S_0} | K_S(x,y)|^2 dxdy.
\end{align*}
The  following properties can be easily checked.
\begin{itemize}
\item $\sum_{n\in \Z^2} \widehat{\omega}(n) = 0$.
\item $\widehat{\omega}( \varepsilon_1 n_1, \varepsilon_2 n_2)= \widehat{\omega}(n_1,n_2)$, where $\varepsilon_1, \varepsilon_2 \in \{\pm 1\}.$
\item there exist $c, C>0$, such that $$ \frac{c}{(1+ n_1^2) (1 + n_2^2)} \le | \widehat{\omega}(n_1, n_2)| \le \frac{C}{(1+n_1^2)(1+n_2^2)}.$$
\end{itemize}
For instance, $\sum_{n\in \Z^2} \widehat{\omega}(n) = 0$ follows from Proposition \ref{prop-rep}. 
These properties combined with Theorem  \ref{FM} yield that 
\begin{itemize}
\item $\omega(0,0)= 0$.
\item $\omega(\varepsilon_1\theta_1, \varepsilon_2\theta_2) = \omega(\theta_1, \theta_2)$ for any $\varepsilon_1, \varepsilon_2 \in \{\pm1\}$ and $\theta_1, \theta_2 \in (-1/2, 1/2)$.
\item the function $\omega(\theta_1, \theta_2)$ is in the multiple Zygmund class $\Lambda_{*}(1,1)$.
\end{itemize}
Hence there exists $C>0$, such that 
\begin{align}\label{zygmund}
| \omega(\theta_1, \theta_2) - \omega(\theta_1, 0) - \omega(0, \theta_2)| \le C | \theta_1\theta_2|. 
\end{align}

\begin{lem}\label{lem-in} There exists $c_1>0$, such that 
\begin{align}\label{claim}
\omega(\theta_1, 0) \ge c_1 | \theta_1| \an \omega(0, \theta_2) \ge c_1 | \theta_2|. 
\end{align}
\end{lem}

Let us postpone the proof of Lemma \ref{lem-in} and proceed to the proof of Lemma \ref{lem-lower}. The inequalities \eqref{zygmund} and \eqref{claim} imply that 
\begin{align*}
\omega(\theta_1, \theta_2) \ge c_1 (| \theta_1| + | \theta_2|) - C| \theta_1\theta_2|.
\end{align*}
Now  if $|\theta_1|$ is small enough such that $ 2 C | \theta_1|  \le c_1$, then we have 
$$
\omega(\theta_1, \theta_2) \ge \frac{c_1}{2}(| \theta_1| + | \theta_2|).
$$
If $2 C | \theta_1|  \ge c_1$, by Corollary \ref{cor-cont-0}, the function $\omega(\theta_1, \theta_2)$ is continuous on $[-\frac{1}{2}, \frac{1}{2}]^2$ and vanishes only at $(0, 0)$. Consequently, 
$$
\inf_{|\theta_1| \ge c_1/2C} \omega(\theta_1, \theta_2)  = c_2 >0.
$$ 
It follows,  by using the elementary fact that $|\theta_1| + |\theta_2|\le 1$, that
$$
\inf_{|\theta_1| \ge c_1/2C} \omega(\theta_1, \theta_2) = c_2 \ge \frac{c_2}{2} (|\theta_1| + |\theta_2|). 
$$
Taking $c= \min(\frac{c_1}{2},\frac{c_1}{2})$, we get the desired inequality \eqref{low-bdd}. 
\end{proof}

Now let us turn to the proof of Lemma \ref{lem-in}. 
\begin{proof}[Proof of Lemma \ref{lem-in}]
 By symmetry, it suffices to prove that there exists $c>0$, such that $\omega(\theta_1, 0) \ge  c | \theta_1|$. To this end, let us denote $\omega_1(\theta_1): = \omega(\theta_1,0)$. Then $\omega_1(0) = 0$ and there exists $c>0$ such that if $k\ne 0$, then
\begin{align*}
\widehat{\omega}_1(k) < 0 \an  |\widehat{\omega}_1(k)| \ge c/(1+k^2).
\end{align*}
Indeed, we have 
\begin{align*}
\omega_1(\theta_1)=  \sum_{k\in \Z}\sum_{n_2\in \Z} \widehat{\omega}(k, n_2) e^{i 2 \pi k \theta_1}.
\end{align*}
If $k\ne 0$, then by \eqref{neg-non-ori}, we have $\widehat{\omega}(k, n_2) < 0$ and hence
\begin{align}\label{low-k}
|\widehat{\omega}_1(k)|=  \sum_{n_2\in \Z}|\widehat{\omega}(k, n_2)| \ge \sum_{n_2\in\Z} \frac{c}{(1+n_2^2)(1+k^2)} \ge \frac{c'}{1+k^2}.
\end{align}
We claim that  $\omega_1(0)= 0$. Indeed, by definition, we have 
\begin{align*}
\omega_1(0) = \sum_{k\in \Z}\sum_{n_2\in \Z} \widehat{\omega}(k, n_2)  = \omega(0, 0) =0,
\end{align*}
where in the last equality, we used Corollary \ref{cor-cont-0} that  claims $\omega(0, 0) =0$. 
Now we have 
\begin{align*}
\sum_{k\in \Z} \widehat{\omega}_1(k) = \omega_1(0) = 0.
\end{align*}
It follows that 
\begin{align*}
\omega_1(\theta_1) &  =   \sum_{k\in\Z} \widehat{\omega}_1(k) e^{i 2\pi k \theta_1} = \sum_{k\in \Z} \widehat{\omega}_1(k) (\frac{e^{i 2\pi k \theta_1} + e^{- i 2\pi k \theta_1} }{2}-1) 
\\ 
&  = \sum_{k\in \Z, k \ne 0} -\widehat{\omega}_1(k) ( 1 - \cos (2\pi k \theta_1)) = \sum_{k\in \Z, k \ne 0} |\widehat{\omega}_1(k)| ( 1 - \cos (2\pi k \theta_1)). 
\end{align*}
Since $ |\widehat{\omega}_1(k)| ( 1 - \cos (2\pi k \theta_1))$ is non-negative for any $k\in\Z$, we have 
\begin{align*}
\omega_1(\theta_1) \ge \sum_{j=1}^{\infty} |\widehat{\omega}_1(2j-1)| ( 1 - \cos (2\pi (2j-1) \theta_1)).
\end{align*}
The inequality \eqref{low-k} implies that there exists $c''>0$, such that   $|\widehat{\omega}_1(2j-1)| \ge \frac{c''}{(2j-1)^2} $, hence we obtain that 
\begin{align*}
\omega_1(\theta_1) & \ge  c'' \sum_{j=1}^\infty \frac{1}{(2j-1)^2} ( 1 - \cos(2 \pi (2j-1)\theta_1)).
\end{align*}
Combining with the Fourier series of the absolutely value function (the Fourier coefficient of the absolute value function on $(-\frac{1}{2}, \frac{1}{2})$ can be computed explicitly): 
$$
 |\alpha | = \frac{1}{4}- \frac{2}{\pi^2} \sum_{j= 1}^\infty\frac{\cos(2\pi (2j-1) \alpha)}{(2j-1)^2}, \text{\, for $\alpha\in (-1/2, 1/2)$;}
$$
$$
\sum_{j=1}^\infty \frac{1}{(2j-1)^2}  = \frac{\pi^2}{8} \text{(take $\alpha =0$ in the above series)}, 
$$
we obtain that 
\begin{align*}
\omega_1(\theta_1) & \ge  c'' \Big(\sum_{j=1}^\infty \frac{1}{(2j-1)^2} -  \sum_{j=1}^\infty \frac{\cos(2 \pi (2j-1)\theta_1)}{(2j-1)^2} \Big)
\\
& = c''\Big(\frac{\pi^2}{8} + \frac{\pi^2}{2} (|\theta_1| - \frac{1}{4})\Big) =  c'' \frac{\pi^2}{2}|\theta_1|.
\end{align*}
The proof of Lemma \ref{lem-in} is complete.
\end{proof}

\begin{proof}[Proof of Proposition \ref{counter-ex}]
By Lemma \ref{K}, it suffices to show that 
\begin{align}\label{finite-int}
\int_{\T^2} \omega_{N^{(1)}}^{-1} dm< \infty.
\end{align}
By Lemma \ref{lem-lower}, the inequality \eqref{finite-int} follows from the following elementary inequality
\begin{align*}
\int_{| \theta_1|< 1/2, | \theta_2| < 1/2} \frac{1}{| \theta_1| + | \theta_2|} d\theta_1 d\theta_2 < \infty.
\end{align*}
\end{proof}

{\bf{Acknowledgements.}}
We would like to thank  Evgeny V. Abakumov, Guillaume Aubrun  and Nikolai K. Nikolskii for many valuable discussions. We are deeply grateful to the anonymous referees
for the careful reading of the manuscript and many helpful suggestions concerning presentation.
A. Bufetov and Y. Qiu are supported by A*MIDEX project (No. ANR-11-IDEX-0001-02), financed by Programme ``Investissements d'Avenir'' of the Government of the French Republic managed by the French National Research Agency (ANR). 

The research of A. Bufetov on this project has received funding from the European Research Council (ERC) under the European Union's Horizon 2020 research and innovation programme under grant agreement No 647133 (ICHAOS). It has also been funded by the Grant MD 5991.2016.1 of the President of the Russian Federation and by  the Russian Academic Excellence Project `5-100'.


\def\cprime{$'$} \def\cydot{\leavevmode\raise.4ex\hbox{.}}

\end{document}